\documentclass{amsart} 

\usepackage{amsmath, amssymb, amsfonts, tikz, hyperref, comment, url, todonotes}
\usetikzlibrary{math}
\usetikzlibrary{decorations.pathreplacing}

\newtheorem{theorem}{Theorem}[section]
\newtheorem{lemma}[theorem]{Lemma}
\newtheorem{proposition}[theorem]{Proposition}
\newtheorem{corollary}[theorem]{Corollary}
\theoremstyle{definition}
\newtheorem{example}[theorem]{Example}

\newtheorem{definition}[theorem]{Definition}
\newtheorem*{definition*}{Definition}
\newtheorem*{theorem*}{Theorem}
\newtheorem*{proposition*}{Proposition}
\newtheorem*{corollary*}{Corollary}

\newtheorem{conjecture}[theorem]{Conjecture}

\newcommand{\Z}{\mathbb{Z}}
\renewcommand{\P}{\mathcal{P}}
\newcommand{\sig}{\mathrm{sig}}

\author{Yibo Gao}
\address{Department of Mathematics, Massachusetts Institute of Technology, \mbox{Cambridge, MA 02139}}
\email{\href{mailto:gaoyibo@mit.edu}{{\tt gaoyibo@mit.edu}}}

\author{Kaarel H\" anni}
\address{Department of Mathematics, Massachusetts Institute of Technology, \mbox{Cambridge, MA 02139}}
\email{\href{mailto:kaarelh@mit.edu}{{\tt kaarelh@mit.edu}}}

\begin{document}
\title{Counting Signed Vexillary Permutations}
\date{\today}

\begin{abstract}
We show that the number of signed permutations avoiding 1234 equals the number of signed permutations avoiding 2143 (also called vexillary signed permutations), resolving a conjecture by Anderson and Fulton. The main tool that we use is the generating tree developed by West. Many further directions are mentioned in the end.
\end{abstract}
\maketitle

\section{Introduction}\label{sec:intro}
Permutation pattern avoidance has been a popular line of research for many years. Denote the symmetric group on $n$ elements by $S_n$. We say that a permutation $w\in S_n$ \textit{avoids} a pattern $\pi\in S_k$ if there do not exist indices $1\leq a_1<\cdots<a_k\leq n$ such that $w(a_i)<w(a_j)$ if and only if $\pi(i)<\pi(j)$. Let $S_n(\pi)$ denote the set of permutations $w\in S_n$ that avoid $\pi$. Two permutations $\pi,\pi'$ are called \textit{Wilf equivalent} if $|S_n(\pi)|=|S_n(\pi')|$ for all $n$. The study of the growth rate of $|S_n(\pi)|$ and the study of nontrivial Wilf equivalence classes have been fruitful. One of the most famous examples of Wilf equivalence classes is all permutation patterns of length 3. Specifically, for any $\pi\in S_3$, $|S_n(\pi)|=C_n$, the $n^{th}$ Catalan number.

Likewise, the set of permutations avoiding 1234 and the set of permutations avoiding 2143 have been traditionally well-studied and enjoy nice combinatorial properties. Their permutation matrices are shown in Figure~\ref{fig:matrix1234and2143}. 
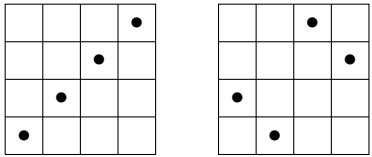
\begin{figure}[h!]
\centering
\begin{tikzpicture}[scale=0.5]
\draw(0,0)--(4,0)--(4,4)--(0,4)--(0,0);
\draw(1,0)--(1,4);
\draw(2,0)--(2,4);
\draw(3,0)--(3,4);
\draw(0,1)--(4,1);
\draw(0,2)--(4,2);
\draw(0,3)--(4,3);
\node at (0.5,0.5) {$\bullet$};
\node at (1.5,1.5) {$\bullet$};
\node at (2.5,2.5) {$\bullet$};
\node at (3.5,3.5) {$\bullet$};
\end{tikzpicture}
\qquad
\begin{tikzpicture}[scale=0.5]
\draw(0,0)--(4,0)--(4,4)--(0,4)--(0,0);
\draw(1,0)--(1,4);
\draw(2,0)--(2,4);
\draw(3,0)--(3,4);
\draw(0,1)--(4,1);
\draw(0,2)--(4,2);
\draw(0,3)--(4,3);
\node at (1.5,0.5) {$\bullet$};
\node at (0.5,1.5) {$\bullet$};
\node at (2.5,3.5) {$\bullet$};
\node at (3.5,2.5) {$\bullet$};
\end{tikzpicture}
\caption{Permutations 1234 and 2143.}
\label{fig:matrix1234and2143}
\end{figure}

A permutation $w$ avoids 1234 if and only if its shape under RSK has at most 3 columns, by Greene's theorem (see for example \cite{stanley1999ec2}). With some further tools in the theory of symmetric functions, 1234-avoiding permutations can be enumerated as follows: $$|S_n(1234)|=\frac{1}{(n+1)^2(n+2)}\sum_{j=0}^n{2j\choose j}{n+1\choose j+1}{n+2\choose j+1}.$$ This enumeration appeared in many previous works including \cite{gessel1990symmetric}, \cite{gessel1998lattice} and \cite{bousquet2002four} and is now an exercise in chapter 7 of \cite{stanley1999ec2}. A permutation that avoids 2143 is called \textit{vexillary}. Vexillary permutations can also be characterized as the permutations whose Rothe diagram, up to a permutation of its rows and columns, is the diagram for a partition \cite{manivel2001symmetric}. Moreover, their associated Schubert polynomials are flag Schur functions \cite{manivel2001symmetric}. West \cite{west1995generating} showed that $|S_n(1234)|=|S_n(2143)|$ for any $n$, so 1234 and 2143 are Wilf equivalent.

The notion of Wilf equivalence can be naturally generalized to signed permutations. The signed permutation group $B_n\simeq(\Z/2\Z)\wr S_n\simeq(\Z/2\Z)^n\rtimes S_n$, also known as the Weyl group of type $B_n$ or $C_n$, consists of permutations $w$ on $\{-n,\ldots,-1,1,\ldots,n\}$ such that $w(i)=-w(-i)$ for all $i\in\{-n,\ldots,-1,1,\ldots,n\}$. We say that $w\in B_n$ \textit{avoids} $\pi\in S_k$ if the natural embedding of $w$ into $S_{2n}$ avoids $\pi$ in the sense of permutation pattern avoidance. For example, $w\in B_2$ given by $w(-2)=1$, $w(-1)=-2$, $w(1)=2$, $w(2)=-1$ contains $231$ and does not contain $123$, as the natural embedding $B_2\to S_{4}$ sends $w$ to $3142\in S_4$, and $3142$ contains $231$ and does not contain $123$. As a warning, this definition of pattern containment is \textbf{not} equivalent to a Weyl group element $w$ of type $B$ avoiding a type $A$ pattern, in the sense of root system pattern avoidance defined by Billey and Postnikov \cite{billey2005smoothness} to study the smoothness of Schubert varieties. 

In particular, let us define
\begin{align*}
B_n(1234)=\{&w\in B_n\ |\ \text{there do not exist}\ -n\leq a<b<c<d\leq n\ \\
&\text{such that}\ w(a)<w(b)<w(c)<w(d)\},\\
B_n(2143)=\{&w\in B_n\ |\ \text{there do not exist}\ -n\leq a<b<c<d\leq n\ \\
&\text{such that}\ w(b)<w(a)<w(d)<w(c)\},
\end{align*}
i.e., the set of signed permutations avoiding 1234 and the set of signed permutations avoiding 2143 respectively, which are the main objects of interest in this paper.

Analogously, $B_n(1234)$ and $B_n(2143)$ have very nice properties. In particular, the enumeration result
$$|B_n(1234)|=\sum_{j=0}^n{n\choose j}^2C_j$$
where $C_j={2j\choose j}/(j+1)$ is the $j^{th}$ Catalan number, is given by Egge \cite{egge2010enumerating}, using techniques involving RSK and jeu-de-taquin. Geometric and combinatorial properties of signed permutations avoiding 2143, which are also called \textit{vexillary signed permutations}, are studied by Anderson and Fulton \cite{anderson2018vexillary}. They conjectured that $|B_n(1234)|=|B_n(2143)|$. The main result of this paper is to answer this conjecture positively.

In fact, there are more similarities between the structures of signed permutations avoiding 1234 and signed permutations avoiding 2143. For $0\leq j\leq n$, and $\pi\in\{1234,2143\}$, define $$B_n^j(\pi):=\{w\in B_n(\pi)\ |\ w(i)>0\ \text{for exactly }j\ \text{indices }i\in \{1,\ldots,n\}\}.$$
\begin{theorem}\label{thm:main}
For $j\leq n$, $|B_n^j(1234)|=|B_n^j(2143)|$.
\end{theorem}

By summing over $j\in [n]$, we obtain the aforementioned conjecture by Anderson and Fulton \cite{anderson2018vexillary} as a corollary.
\begin{corollary}\label{cor:main}
For $n\in\Z_{\geq1}$, \[|B_n(2143)|=|B_n(1234)|=\sum_{j=0}^n{n\choose j}^2C_j.\]
\end{corollary}

As pointed out by Christian Gaetz, Theorem~\ref{thm:main} also implies a direct analogue of Corollary \ref{cor:main} for type $D$. Recall that the Weyl group of type $D_n$ can be realized as an order 2 subgroup of the Weyl group of type $B_n$. Specifically,
$$D_n:=\{w\in B_n\ |\ w(i)<0\text{ for an even number of indices }i\in\{1,\ldots,n\}\}.$$
We then say that $w\in D_n$ avoids a pattern $\pi$ if $w\in D_n\subset B_n$ avoids $\pi$. And as an analogous notation, let $D_n(\pi)$ denote the set of $w\in D_n$ that avoids $\pi$. Similarly, elements in $D_n(2143)$ are called \textit{vexillary} in type $D$ \cite{billey1998vexillary}. By summing the equality in Theorem \ref{thm:main} over $j\in [n]$ with $n-j$ even, we obtain the analogous enumeration result in type $D$.
\begin{corollary}
For $n\in\mathbb{Z}_{\geq1}$, $|D_n(1234)|=|D_n(2143)|$.
\end{corollary}

The main tool that we use in the proof of Theorem~\ref{thm:main} is the idea of generating trees developed by West \cite{west1995generating} to show that $|S_n(1234)|=|S_n(1243)|=|S_n(2143)|$. A generating tree is a rooted labeled tree for which the label at a vertex determines its descendants (their number and their labels). The generating trees considered by West have vertices that correspond to permutations avoiding a fixed pattern $\pi$. The descendants of a vertex corresponding to the permutation $w\in S_n$ correspond to permutations formed by inserting a new largest element $n+1$ to some location in $w$ (so that $\pi$ is still avoided). The usefulness of such generating trees stems in part from the fact that it is often possible to present an isomorphic tree with vertices labeled by only a few permutation statistics, with a simple enough succession rule to be fit for further analysis. In the case of $S_n(1234)$ versus $S_n(2143)$, West was able to find a simple description of both trees and observed that the two are naturally isomorphic, thus proving $|S_n(1234)|=|S_n(2143)|$ bijectively. 

There are two main difficulties in proving the simple-looking theorem (Theorem~\ref{thm:main}). First, as pointed out by Anderson and Fulton \cite{anderson2018vexillary}, the bijection between $S_n(1234)$ and $S_n(2143)$ provided by West \cite{west1995generating} does not preserve whether the permutation equals its reverse complement or not, suggesting a more careful choice of statistics for the generating trees, described in Section~\ref{sec:suc}. Second, the generating trees for $B_n^j(1234)$ and $B_n^j(2143)$ turn out to be far from isomorphic unlike the case of $S_n(1234)$ versus $S_n(2143)$ so we finish the proof by using certain generating functions in Section~\ref{sec:proof}.
We end in Section~\ref{sec:open} with discussion on open problems.
\section{Generating trees for 1234 and 2143 avoiding permutations}\label{sec:suc}

We will start working towards an explicit generating tree for $B^j_n(1234)$ and $B^j_n(2143)$. Throughout the section, a signed permutation $w\in B_n$ should be visualized by a point graph, where the $x$-axis corresponds to the indices ${-}n,{-}n{-}1,\ldots,{-}1$, $1,\ldots,n$, and the $y$-axis corresponds to the images $w({-}n),w({-}n{-}1),\ldots,w(n)$. As a one-line notation, we will denote $w$ by $[w({-}n),w({-}n{-}1),\ldots,w({-}1)]$ in a nonstandard way, for reasons soon to be clear. A visualization of $w\in B_n$ is shown in Figure~\ref{fig:pointgraph}.
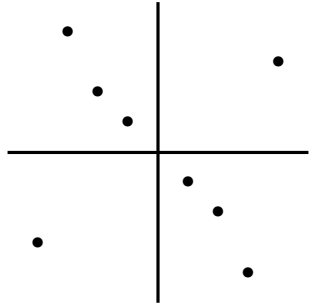
\begin{figure}[h!]
\centering
\begin{tikzpicture}[scale=0.4]
\draw[very thick](0,-5)--(0,5);
\draw[very thick](-5,0)--(5,0);
\node at (-4,-3) {$\bullet$};
\node at (-3,4) {$\bullet$};
\node at (-2,2) {$\bullet$};
\node at (-1,1) {$\bullet$};
\node at (1,-1) {$\bullet$};
\node at (2,-2) {$\bullet$};
\node at (3,-4) {$\bullet$};
\node at (4,3) {$\bullet$};
\end{tikzpicture}
\caption{Visualization for $[-3,4,2,1].$}
\label{fig:pointgraph}
\end{figure}

We first prove some simple lemmas regarding structures of signed permutations avoiding 2143 or 1234.
\begin{lemma}\label{lem:struc}
The following statements are true:
\begin{enumerate}
    \item Suppose $w\in B_n^j(2143)$, where the $j$ positive indices with positive images are $1\leq i_1<i_2<\ldots<i_j\leq n$. Then $w(i_1)<w(i_2)<\ldots <w(i_j)$. 

    \item Suppose $w\in B_n^j(1234)$, where the $j$ positive indices with positive images are $1\leq i_1<i_2<\ldots<i_j\leq n$. Then $w(i_1)>w(i_2)>\ldots>w(i_j)$.
\end{enumerate}
\end{lemma}

\begin{proof}
The proofs for the two cases are completely analogous so let's consider $w\in B_n^j(2143)$. If (1) were not true, then there exists a pair of positive indices $i<j$ with $1\leq w(j)<w(i)\leq n$. Consider the pattern forming at the indices $-j, -i, i, j$. We have $w(-i)<w(-j)<0<w(j)<w(i)$, or in other words, the pattern is $2143$. This is a contradiction with $w\in B_n^j(2143)$.
\end{proof}

\begin{lemma}\label{lem:ignore}
Let $\pi\in \{2143, 1234\}$. If $w\in B_n$ contains $\pi$, then there exist indices ${-}n\leq i_1<i_2<i_3<i_4\leq n$ forming the pattern $\pi$ so that there exists no $i_k>0$ with $w(i_k)<0$.
\end{lemma}
\begin{proof}
Let $w\in B_n$ and $\pi\in\{2143,1234\}$. Say $w$ contains $\pi$ at indices $i_1<i_2<i_3<i_4$. If there exists no $i_k>0$ with $w(i_k)<0$, then we are done. And if there exists no $i_k<0$ with $w(i_k)>0$, then we are also done by considering indices $-i_4<-i_3<-i_2<-i_1$ as $\pi$ equals its reverse complement. So we can without loss of generality assume that there exists $i_k>0$ with $w(i_k)<0$ and there exists $i_{\ell}<0$ with $w(i_{\ell})>0$. Here, $\ell<k$.

If $\pi=1234$, this scenario is impossible since $i_{\ell}<0<i_k$ but $w(i_{\ell})>0>w(i_k)$. If $\pi=2143$, then either we have $\ell=1$, $k=2$, in which case $0<i_2<i_3<i_4$, $0<w(i_1)<w(i_4)<w(i_3)$ and the indices $-i_4<-i_3<i_3<i_4$ form 2143 or we have $\ell=3$, $k=4$, in which case $i_1<i_2<i_3<0$, $w(i_2)<w(i_1)<w(i_4)<0$ and the indices $i_1<i_2<-i_2<-i_1$ form 2143 with the desired property.
\end{proof}

Lemma~\ref{lem:struc} is saying that signed permutations that avoid 2143 (or 1234) have increasing (decreasing) sequence in the top right and bottom left quadrant, and moreover, Lemma~\ref{lem:ignore} allows us to ignore the contribution from the bottom right quadrant so that we can focus on the top left quadrant. Figure~\ref{fig:struct} depicts this idea.

\begin{figure}[h!]
\centering
\begin{tikzpicture}[scale=0.3]
\draw[very thick](0,-5)--(0,5);
\draw[very thick](-5,0)--(5,0);
\node at (-4,-4) {$\bullet$};
\node at (-3,-3) {$\bullet$};
\node at (-2,-2) {$\bullet$};
\node at (-1,-1) {$\bullet$};
\node at (1,1) {$\bullet$};
\node at (2,2) {$\bullet$};
\node at (3,3) {$\bullet$};
\node at (4,4) {$\bullet$};
\node at (5.5,2.5) {$j{=}4$};
\draw[thick](1,-1)--(4,-4);
\draw[thick](4,-1)--(1,-4);
\end{tikzpicture}
\hspace{1cm}
\begin{tikzpicture}[scale=0.3]
\draw[very thick](0,-5)--(0,5);
\draw[very thick](-5,0)--(5,0);
\node at (-4,-1) {$\bullet$};
\node at (-3,-2) {$\bullet$};
\node at (-2,-3) {$\bullet$};
\node at (-1,-4) {$\bullet$};
\node at (1,4) {$\bullet$};
\node at (2,3) {$\bullet$};
\node at (3,2) {$\bullet$};
\node at (4,1) {$\bullet$};
\node at (5.5,3) {$j{=}4$};
\draw[thick](1,-1)--(4,-4);
\draw[thick](4,-1)--(1,-4);
\end{tikzpicture}
\caption{Structure of signed permutations avoiding 2143 (left) or 1234 (right)}
\label{fig:struct}
\end{figure}
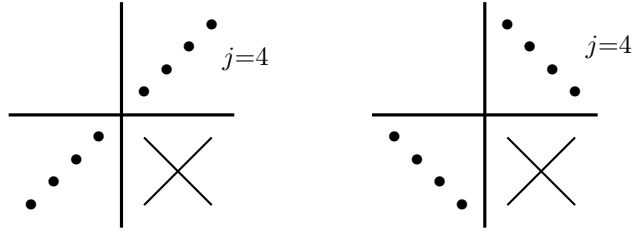

\begin{example}

Figure~\ref{fig:example-perm} shows an explicit signed permutation that avoids $2143$, with the dots in the bottom left and top right quadrant highlighted, and the dots in the bottom right quadrant in gray. Lemma~\ref{lem:struc} tells us that the highlighted dots form an increasing sequence. Lemma~\ref{lem:ignore} tells us that avoidance of $2143$ can be checked without considering the gray dots.

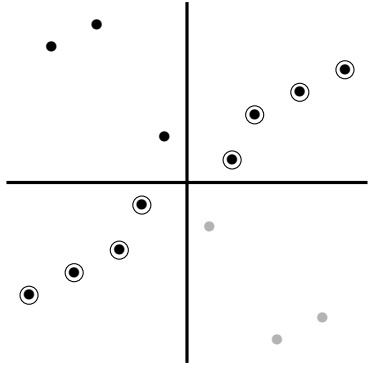
\begin{figure}[h!]
\begin{tikzpicture}[scale=0.3]
\draw[very thick](0,-8)--(0,8);
\draw[very thick](-8,0)--(8,0);
\node at (-7,-5) {$\bullet$};
\draw (-7,-5) circle (0.4);
\node at (-5,-4) {$\bullet$};
\draw (-5,-4) circle (0.4);
\node at (-3,-3) {$\bullet$};
\draw (-3,-3) circle (0.4);
\node at (-2,-1) {$\bullet$};
\draw (-2,-1) circle (0.4);
\node at (-6,6) {$\bullet$};
\node at (-4,7) {$\bullet$};
\node at (-1,2) {$\bullet$};
\node at (7,5) {$\bullet$};
\draw (7,5) circle (0.4);
\node at (5,4) {$\bullet$};
\draw (5,4) circle (0.4);
\node at (3,3) {$\bullet$};
\draw (3,3) circle (0.4);
\node at (2,1) {$\bullet$};
\draw (2,1) circle (0.4);
\node [opacity=.3] at (6,-6) {$\bullet$};
\node [opacity=.3] at (4,-7) {$\bullet$};
\node [opacity=.3] at (1,-2) {$\bullet$};
\end{tikzpicture}
\caption{The $2143$-avoiding signed permutation $[-5,6,-4,7,-3,-1,2]$}
\label{fig:example-perm}
\end{figure}
\end{example}

With Lemma~\ref{lem:struc} and Lemma~\ref{lem:ignore}, we will generate $w\in B_n^j(\pi)$ by inserting elements into the top left quadrant one by one in the increasing order of their images. This idea of generating permutations is developed by West \cite{west1995generating}. For $w\in B_n^j(\pi)$, call the spaces between negative indices \textit{sites}, including the space left of $-n$ and the space right of $-1$ so that there are $n+1$ sites. Similarly, call the vertical spaces \textit{gaps}. To define insertions more formally, we introduce the auxiliary function
\[
\beta_\ell(x)= 
\begin{cases} 
x & |x|<\ell, \\
x-1 & x<-\ell,\\
x+1 & x>\ell,\\
\end{cases}
 \]
which can be thought as the function pushing images to their new locations when the gap between $\ell-1$ and $\ell$ gets a new image.
\begin{definition}\label{def:insert}
For $w\in B_n$, let $w_{\ell}^{-i}\in B_{n+1}$ be the signed permutation obtained by \textit{inserting a new element to site} $-i$ \textit{and gap} $\ell$, defined by
\[
w_{\ell}^{-i}(-k)=
\begin{cases}
\beta_{\ell}\big(w(-k-1)\big) & i+1\leq k\leq n+1,\\
\ell & k=i,\\
\beta_{\ell}\big(w(-k)\big) & 1\leq k\leq i-1.
\end{cases}
\]
\end{definition}

Definition~\ref{def:insert} is just a formal way to express inserting an element to a specific position (and its antipode) in a signed permutation. Now we are ready to introduce the main object of interest in this section.
\begin{definition}
For $\pi\in\{2143,1234\}$ and $j\geq0$, let $BT^j(\pi)$, the \textit{signed permutation pattern avoidance tree}, to be a rooted tree labeled by signed permutations, defined as follows:
\begin{itemize}
    \item its root is $[-j,\ldots,-1]$ for $\pi=2143$ and $[-1,\ldots,-j]$ for $\pi=1234$,
    \item the successors of $w\in B_n^j(\pi)$ are all $w_{\ell}^{-i}$'s with $1\leq i\leq n+1$ and $m<\ell\leq n+1$ that still avoid $\pi$, where $m=\max_{1\leq k\leq n}\{w(-k)\}\cup\{0\}$.
\end{itemize}
\end{definition}

Let us briefly discuss some essential properties of this tree. Note that any permutation avoiding $\pi$ can be constructed by starting with the signed permutation in its top right and bottom left quadrants (which has the form of the corresponding root by Lemma \ref{lem:struc}) and inserting its other elements in increasing order to the top left quadrant. Hence, any permutation avoiding $\pi$ appears in $BT^j(\pi)$. Moreover, any permutation avoiding $\pi$ appears in $BT^j(\pi)$ exactly once, since there is only one way to insert its elements in increasing order. Hence, the vertices of $BT^j(\pi)$ which are $n-j$ steps away from the root correspond precisely to the signed permutations of length $n$ avoiding $\pi$ with exactly $j$ positive indices with positive images. In other words, the vertices $n-j$ steps away from the root correspond to $B^j_n(\pi)$.

The main results of this section are the following.
\begin{proposition}\label{prop:suc2143}
The generating tree given by the following:
\begin{itemize}
    \item the label of the root is $(j{+}1, j{+}1, j{+}1)$;
    \item the succession function $suc$ that takes a label as its input and outputs the set of successors is defined recursively as follows:
    $$suc(x,y,z)=
    \begin{cases}
    \emptyset &z=0,\\
    \{(2,y{+}1,z),(3,y{+}1,z),\ldots,(x{+}1,y{+}1,z)\\
    (x,x{+}1,z),(x,x{+}2,z),\ldots,(x,y,z)\}\bigcup suc(x,x,z{-}1) &z\geq1.
    \end{cases}$$
\end{itemize}

is isomorphic as a rooted tree to $BT^j(2143)$.
\end{proposition}

\begin{proposition}\label{prop:suc1234}
The generating tree given by the following:
\begin{itemize}
    \item the label of the root is $(j{+}1, j{+}1, j{+}1)$;
    \item the succession function $suc$ that takes a label as its input and outputs the set of successors is defined recursively as follows:
    $$suc(x,y,z)=
    \begin{cases}
    \{(2,y{+}1,z),(3,y{+}1,z),\ldots,(x{+}1,y{+}1,z)\\
    (x,x{+}1,z),(x,x{+}2,z),\ldots,(x,y,z)\} &z=1,\\
    \{(2,y{+}1,z),(3,y{+}1,z),\ldots,(x{+}1,y{+}1,z)\}\\
    \bigcup suc(x,y,z{-}1) &z\geq2.
    \end{cases}$$
\end{itemize}
is isomorphic as a rooted tree to $BT^j(1234)$.
\end{proposition}

Before proving Proposition~\ref{prop:suc2143} and Proposition~\ref{prop:suc1234}, we first discuss some important statistics on the signed permutations of interest, that are related to the variables $x,y,z$ in the above propositions. Examples will come shortly.

For $w\in B^j_n(\pi)$, a \textit{site before the first ascent (descent)} is defined to be a site such that elements to the left of this site are decreasing (increasing). In particular, if $w(-n)>\cdots>w(-1)$ (or increasing), then there are $n+1$ sites before the first ascent (descent). The number of sites before the first ascent (descent) is usually denoted via the variable $x$.

For $w\in B^j_n(\pi)$, an \textit{active site} with respect to a fixed gap $\ell$, is a site such that inserting into this site and gap $\ell$ results in a signed permutation that avoids $\pi$. The number of active sites is usually denoted $y$. We will make further specifications for $\pi=2143$ and $\pi=1234$.

For $w\in B_n^j(\pi)$, there are $j$ positive indices $i$ with positive images. In other words, there are $j$ elements in the top right quadrant, and they divide the top left quadrant into horizontal ``layers". Formally, the \textit{layer number} is $1+\#\{i>0:w(i)>\max_{k<0}w(k)\}$, describing the current layer that we are inserting elements into. The layer number is denoted $z$, and it ranges from $j+1$ to 1.

Recall that we are constructing signed permutations in $BT^j(\pi)$ by inserting elements into the top left quadrant in increasing order of the images. Therefore, we will be saying inserting into some layer instead of inserting into some gap. The following Lemma~\ref{lem:activesites} is useful and can be observed directly so we omit the proof.

\begin{lemma}\label{lem:activesites}
Let $w\in B_n^j(\pi)$ and fix a layer $z$ that we are inserting elements into. If we insert the new maximal image $\ell$ in the left quadrants to some active site $-i$, then the new active sites of $w^{-i}_{\ell}$ are a subset of the old ones: here we think of the site where we inserted $\ell$ to have split into two (so the number of active sites may potentially increase by at most 1). Furthermore, if a previously active site becomes inactive, then inserting $\ell+1$ there would create a pattern $\pi$ involving both $\ell$ and $\ell+1$.
\end{lemma}

Now we are ready to separate the cases $\pi=2143$ and $\pi=1234$.

First consider $\pi=2143$. Keep track of the following statistics on $w\in B_n^j(2143)$:
\begin{itemize}
    \item $x$: the number of sites before the first descent,
    \item $y$: the number of active sites in the current layer $z$,
    \item $z$: the layer number, which is the lowest layer to which the maximal image of the negative indices can be inserted.
\end{itemize}
Figure~\ref{fig:2143} shows the statistics $x=3$, $y=5$ and $z=2$ for $w=[-6,4,-3,5,2,1]$ where we see that the layer numbers are decreasing from bottom to top and the active sites in each layer are labeled by $\times$.
\begin{figure}[h!]
\centering
\begin{tikzpicture}[scale=0.6]
\draw[very thick](0,-7)--(0,8);
\draw[very thick](-8,0)--(8,0);
\node at (-1,1) {$\bullet$};
\node at (-2,2) {$\bullet$};
\node at (-5,4) {$\bullet$};
\node at (-3,5) {$\bullet$};
\node at (-4,-3) {$\bullet$};
\node at (-6,-6) {$\bullet$};
\node at (4,3) {$\bullet$};
\node at (6,6) {$\bullet$};
\draw(-7,6)--(6,6);
\draw(-7,3)--(4,3);
\draw[dashed](-6,-7)--(-6,7);
\draw[dashed](-5,-7)--(-5,7);
\draw[dashed](-4,-7)--(-4,7);
\draw[dashed](-3,-7)--(-3,7);
\draw[dashed](-2,-7)--(-2,7);
\draw[dashed](-1,-7)--(-1,7);
\node at (-6.5,5.5) {$\times$};
\node at (-5.5,5.5) {$\times$};
\node at (-4.5,5.5) {$\times$};
\node at (-2.5,5.5) {$\times$};
\node at (-1.5,5.5) {$\times$};
\node at (-6.5,6.5) {$\times$};
\node at (-5.5,6.5) {$\times$};
\node at (-4.5,6.5) {$\times$};
\node at (-8.2,4.5) {$z=$};
\node (z3) at (-7.5,1.5) {3};
\node (z2) at (-7.5,4.5) {2};
\node (z1) at (-7.5,6.5) {1};
\draw[->] (z3)--(z2);
\draw[->] (z2)--(z1);
\node at (1,5.5) {$y=5$};
\node at (5,4.5) {$j=2$};
\draw [decorate,decoration={brace,amplitude=5pt}]
(-7,7) -- (-4,7) node [black,midway,yshift=10pt] {$x=3$};
\end{tikzpicture}
\caption{The statistics $x=3$, $y=5$, $z=2$ for $w=[-6,4,-3,5,2,1]\in B_6^2(2143).$}
\label{fig:2143}
\end{figure}
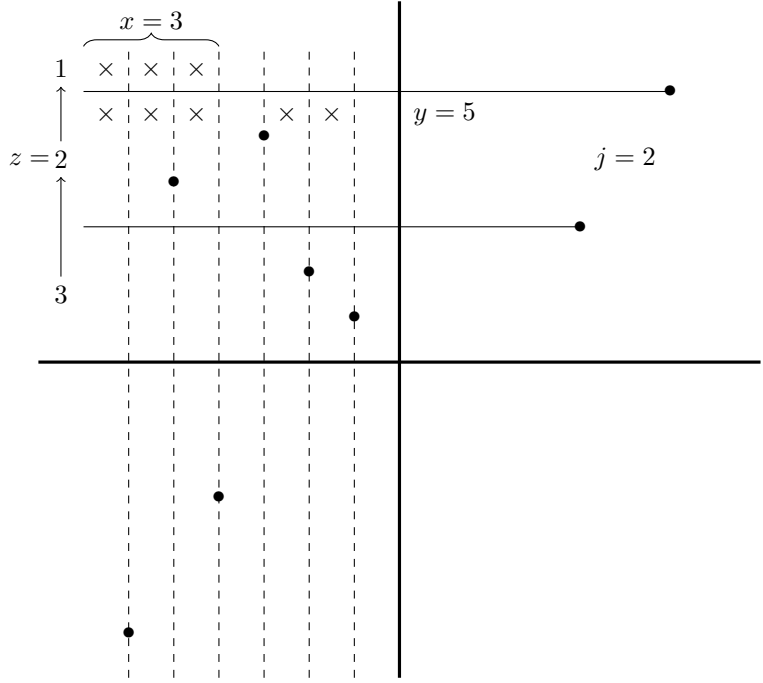

\begin{proof}[Proof of Proposition~\ref{prop:suc2143}]
It suffices to show that if the statistics of $w$ are $x,y,z$, then the multiset of statistics of the successors of $w$ in $BT^j(2143)$ is precisely what is given in the proposition.

As we are inserting new images in increasing order, there are no successors with layer number $z'>z$. Let us first determine the number of active sites $y_{z'}$ in each layer numbered $z'\leq z$. We know that the number of active sites in the layer $z$ is $y$. Let us show that the number of active sites in all layers $z'<z$ is $y_{z'}=x$. To see this, note that all sites before the first descent are active, as inserting there, the inserted element would have to be a $2$ or a $1$ in a $2143$, but this is impossible as all larger elements are in the top right quadrant, hence form an increasing sequence (so there is no $43$). Additionally, note that inserting to a site after the first descent would form a $2143$ involving the first descent as $21$, the inserted element as $4$, and the element immediately below it in the top right quadrant as $3$ (this exists since $z'<z$, so the layer was previously empty). Therefore, the active sites for layers $z'<z$ are precisely the $x$ sites before the first descent.

Let us now determine the successors with layer number $z'$, given that the number of active sites in the layer $z'$ is initially $y_{z'}$. It is again easy to see that all sites before the first descent are active, and it follows from Lemma \ref{lem:activesites} that all active sites remain active after such an insertion. The new first descent appears immediately after the insertion. Hence, the successors from this case are $(2,y_{z'}+1, z'), (3, y_{z'}+1, z'), \ldots,  (x+1,y_{z'}+1, z')$.

Let us now consider the case of inserting to an active site after the first descent. In that case, the insertion leaves the position of the first descent unchanged. As for the active sites, all sites before the first descent are still active, all sites to the right of the first descent but to the left of the insertion are inactive (since inserting there would create an obvious $2143$), and all previously active sites to the right of the insertion remain active, since by Lemma \ref{lem:activesites}, if we also inserted to one such site and created a $2143$, it would need to involve both of the last two insertions, and these could only be $2$ and $3$, but no $4$ can be found between them. Hence, the successors from this case are $(x,x+1,z'), (x, x+2, z'), \ldots, (x, y_{z'}, z')$.

We have now explicitly found the set of successors, since we have found $y_{z'}$ for each $z'\leq z$ and given the successors with each layer number $z'\leq z$ in terms of $y_{z'}$. This multiset of successors is what is given in the proposition.
\end{proof}

Next consider $\pi=1234$. Keep track of the following statistics on $w\in B_n^j(2134)$:
\begin{itemize}
    \item $x$: the number of sites before the first ascent,
    \item $y$: the number of active sites in the top layer,
    \item $z$: the layer number.
\end{itemize}
Figure~\ref{fig:1234} shows the statistics $x=3$, $y=7$, $z=3$ for $w=[2,-3,4,-5,1,-6]$.

\begin{figure}[h!]
\centering
\begin{tikzpicture}[scale=0.6]
\node at (-2,1) {$\bullet$};
\node at (-6,2) {$\bullet$};
\node at (-4,4) {$\bullet$};
\node at (-1,-6) {$\bullet$};
\node at (-3,-5) {$\bullet$};
\node at (-5,-3) {$\bullet$};
\node at (1,6) {$\bullet$};
\node at (3,5) {$\bullet$};
\node at (5,3) {$\bullet$};
\draw[very thick](0,-7)--(0,8);
\draw[very thick](-8,0)--(8,0);
\draw(-7,6)--(1,6);
\draw(-7,5)--(3,5);
\draw(-7,3)--(5,3);
\draw[dashed](-6,-5)--(-6,7);
\draw[dashed](-5,-5)--(-5,7);
\draw[dashed](-4,-5)--(-4,7);
\draw[dashed](-3,-7)--(-3,7);
\draw[dashed](-2,-7)--(-2,7);
\draw[dashed](-1,-7)--(-1,7);
\node at (-6.5,4.5) {$\times$};
\node at (-5.5,4.5) {$\times$};
\node at (-4.5,4.5) {$\times$};
\node at (-6.5,5.5) {$\times$};
\node at (-5.5,5.5) {$\times$};
\node at (-4.5,5.5) {$\times$};
\node at (-6.5,6.5) {$\times$};
\node at (-5.5,6.5) {$\times$};
\node at (-4.5,6.5) {$\times$};
\node at (-3.5,6.5) {$\times$};
\node at (-2.5,6.5) {$\times$};
\node at (-1.5,6.5) {$\times$};
\node at (-0.5,6.5) {$\times$};
\node at (5,5) {$j=3$};
\node at (-3.5,7.5) {$y=7$};
\draw [decorate,decoration={brace,amplitude=5pt}]
(-4,-5.5) -- (-7,-5.5) node [black,midway,yshift=-10pt] {$x=3$};
\node at (-8.2,4) {$z=$};
\node (z3) at (-7.5,4) {3};
\node (z2) at (-7.5,5.5) {2};
\node (z1) at (-7.5,6.5) {1};
\node (z4) at (-7.5,1.5) {4};
\draw[->] (z4)--(z3);
\draw[->] (z3)--(z2);
\draw[->] (z2)--(z1);
\end{tikzpicture}
\caption{The statistics $x=3$, $y=7$, $z=3$ for $w=[2,-3,4,-5,1,-6]\in B_6^3(1234).$}
\label{fig:1234}
\end{figure}
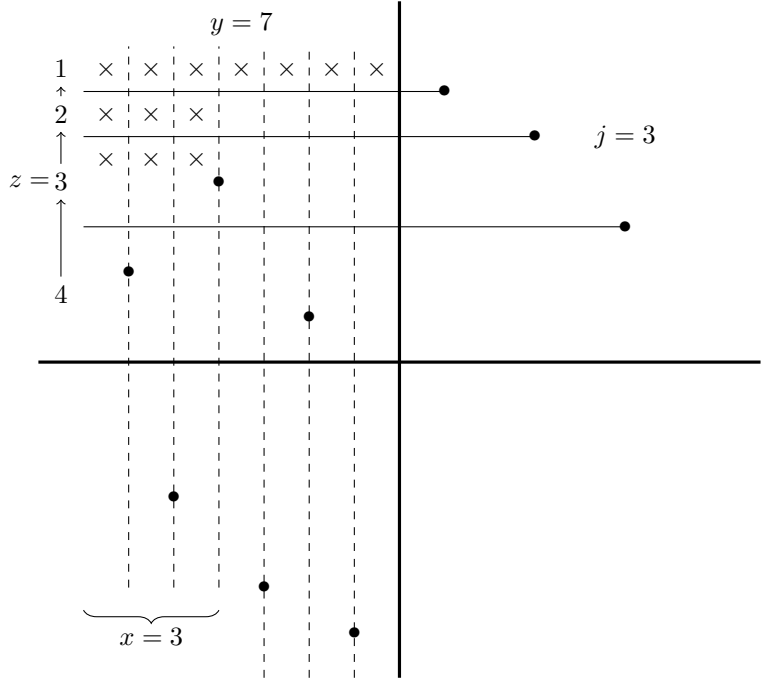

\begin{proof}[Proof of Proposition~\ref{prop:suc1234}]

Again, it suffices to show that if the statistics of $w$ are $x,y,z$, then the multiset of statistics of the successors of $w$ in $BT^j(1234)$ is precisely what is given in the proposition.

As before, there are no successors with layer number $z'>z$. The active sites for layers numbered $z'$ such that $z\geq z'>1$ are precisely the sites before the first ascent, as all sites before the first ascent are active (since the inserted element would have to be a $1$ or $2$ in a $1234$, but then there is no $34$ since all elements larger than the inserted element are in the top right quadrant and form a decreasing sequence), and all sites after the first ascent are inactive because of forming a $1234$ with the maximal element of the top right quadrant serving as a $4$. Also note that for $z>1$, all sites in the top layer are active, as if some site was inactive, there would be a $123$ in the top left quadrant, but this would also mean there exists a $1234$ with the maximal element in the top right quadrant as a $4$, which is impossible. This means that the successors with $z'$ such that $z\geq z'>1$ are precisely $(2,y+1, z'), (3,y+1, z'), \ldots, (x+1, y+1, z')$.

As for successors with $z'=1$, in the top layer all sites before the first ascent are active, and inserting to such a site leaves all active sites active (the proof is analogous to what was seen for Proposition \ref{prop:suc2143}). This case gives the successors $(2, y+1,  1), (3,y+1,1), \ldots, (x+1, y+1, 1)$. If we instead insert to an active site after the first ascent, then the new active sites are precisely the ones to the left of the insertion. The proof is again analogous to what was seen for Proposition \ref{prop:suc2143}. This final case gives the successors $(x, x+1, z), (x, x+2, z), \ldots, (x,y,z)$.

Putting everything together, this multiset of successors is what is given in the proposition.

\end{proof}
 \section{Finishing the proof}\label{sec:proof}
Proposition~\ref{prop:suc2143} and Proposition~\ref{prop:suc1234} allow us to translate the questions of enumerating $B_n^j(2143)$ and $B_n^j(1234)$ to questions of enumerating lattice paths in the integer lattice $\Z^3$ with specified rules. Respectively, let $\P^{2143}$ be the set of all lattice paths specified by the succession rule in Proposition~\ref{prop:suc2143} and let $\P^{1234}$ be the set of all lattice paths specified by the succession rule in Proposition~\ref{prop:suc1234}. We allow arbitrary starting point $(x,y,z)$ for those paths with $2\leq x\leq y$ and $1\leq z$ besides those that start at $(j+1,j+1,j+1)$. We view such a lattice path as a sequence of points connected by edges.

For a path $P\in\P^{2143}$ and an edge $e$ of $P$ that goes from $(x_1,y_1,z_1)$ to $(x_2,y_2,z_2)$, we say that $e$ is \textit{recorded}
\begin{itemize}
    \item if $z_1=z_2$ and $y_2=y_1+1$;
    \item if $z_1>z_2$ (and $y_2=x_1+1$).
\end{itemize}
Notice that if $z_1>z_2$, then we are forced to use the succession rule of $(x_1,x_1,z_2)$ to go to $(x_2,y_2,z_2)$ and thus $y_2=x_1+1$. Analogously, for $P\in\P^{1234}$ and an edge $e$ of $P$ that goes from $(x_1,y_1,z_1)$ to $(x_2,y_2,z_2)$, we say that $e$ is \textit{recorded} if $y_2=y_1+1$. In particular, if $z_2\geq2$, the edge is always recorded.
We see from the succession rule in Section~\ref{sec:suc} that if an edge is not recorded, then the $x$-coordinates are the same for the two points connected by the edge.
\begin{definition}
For a path $P\in\P^{\pi}$, $\pi\in\{1234,2143\}$, define its \textit{signature} $\sig(P)$ to be the tuple consists of the $x$-coordinate of the starting point, appended with the $x$-coordinates of the ending points of recorded edges in order.
\end{definition}
\begin{example}
Consider the following paths $P\in\P^{2143}$ and $P'\in\P^{1234}$ which are
\begin{align*}
    P=&(4,4,3)\rightarrow(3,5,3)\dashrightarrow(3,5,3)\rightarrow(4,4,2)\rightarrow(2,5,2)\dashrightarrow(2,4,2)\\
    &\dashrightarrow(2,4,2)\rightarrow(2,2,1)\rightarrow(2,3,1)\dashrightarrow(2,3,1)\\
    P'=&(4,4,3)\rightarrow(3,5,3)\rightarrow(4,6,3)\rightarrow(2,7,2)\rightarrow(2,8,1)\dashrightarrow(2,7,1)\\
    &\dashrightarrow(2,7,1)\dashrightarrow(2,5,1)\dashrightarrow(2,4,1)\dashrightarrow(2,4,1)\rightarrow(2,5,1)\dashrightarrow(2,4,1)
\end{align*}
where the recorded edges are written as solid arrows and the edges not recorded are written as dashed arrows. Both paths have signature $(4,3,4,2,2,2)$.
\end{example}

The main goal of this section is to show that for a fixed starting point $v$, a fixed signature $\gamma$ and $n\geq1$, the number of paths in $\P^{\pi}$ that start with $v$, have signature $\gamma$ and have length $n$ is the same for $\pi\in\{1234,2143\}$. To do this, let us define the corresponding generating functions. 
For $\pi\in\{1234,2143\}$, $k\geq0$, $q\geq1$, $\gamma=(\gamma_1,\ldots,\gamma_m)\in\Z^m$, define $\P^{\pi}_{k,q,\gamma}$ to be the set of paths in $\P^{\pi}$ that start at $(\gamma_1,\gamma_1+k,q)$ and have signature $\gamma$. For any path $P$, let its length $\ell(P)$ be the number of points that it contains. Notice that if $P$ has signature $\gamma=(\gamma_1,\ldots,\gamma_m)$, then clearly $\ell(P)\geq m$. 
\begin{definition}
For $\pi\in\{1234,2143\}$, $k\geq0$, $q\geq1$, $\gamma\in\Z^m$ with $m\geq1$, define
$$F^{\pi}(k,q,\gamma):=\sum_{P\in\P^{\pi}_{k,q,\gamma}}t^{\ell(P)-m}.$$
\end{definition}
We are going to recursively compute $F^{\pi}(k,q,\gamma)$ and then compare $F^{1234}(k,q,\gamma)$ with $F^{2143}(k,q,\gamma)$. As for some notations, if $\gamma\in\Z^m$, write $|\gamma|=m$. For convention, we say $F^{\pi}(k,q,\gamma)=0$ if $q\leq0$ or $|\gamma|=0$. Let $\gamma'=(\gamma_2,\gamma_3,\ldots,\gamma_m)$, which is $\emptyset$ if $m=1$ and let $\gamma''=(\gamma_3,\ldots,\gamma_m)$, which is $\emptyset$ if $m\leq2$. And we will restrict our attention to only those $\gamma$'s that can be signatures of some valid paths in $\P^{\pi}$. Namely, we require $2\leq\gamma_{i+1}\leq\gamma_i+1$. Finally, for simplicity, let $s=1+t+t^2+\cdots=1/(1-t)$.
\begin{lemma}\label{lem:2143gen}
For $k\geq0$, $q\geq1$, $\gamma\in\Z^m$ with $m\geq1$, we have
\begin{equation*}
F^{2143}(k,q,\gamma)=\begin{cases}
s^k &|\gamma|=1,\\
F^{2143}(0,q{-}1,\gamma)+F^{2143}(\gamma_1{+}1{-}\gamma_2,q,\gamma') &|\gamma|\geq2,k=0,\\
sF^{2143}(k{-}1,q,\gamma)+sF^{2143}(\gamma_1{+}1{-}\gamma_2{+}k,q,\gamma')\\
-sF^{2143}(\gamma_1{-}\gamma_2{+}k,q,\gamma') &|\gamma|\geq2,k\geq1.
\end{cases}
\end{equation*}
\end{lemma}
\begin{proof}
We refer the readers to the succession rule in Proposition~\ref{prop:suc2143}.

If $|\gamma|=1$, then the signature has length 1 and we are summing over paths that start at $(\gamma_1,\gamma_1+k,q)$ with no recorded edges. As soon as we decrease $q$, which is the $z$-coordinate, we need to use the succession rule for $(\gamma_1,\gamma_1,q-1)$ and then every edge is recorded so we cannot have any edges afterwards. Therefore, the only additional points on this path come from any number of $(\gamma_1,\gamma_1{+}k,q)$ followed by any number of $(\gamma_1,\gamma_1{+}k{-}1,q)$ and so on, finally ending with any number of $(\gamma_1,\gamma_1{+}1,q)$. The resulting generating function is then $(1+t+t^2+\cdots)^k=s^k$.

If $k=0$, then our paths start at $(\gamma_1,\gamma_1,q)$. The next edge must be recorded. There are exactly two options: either decrease the $z$-coordinate $q$, or go directly to the next signature value $\gamma_2$ at the same $z$-coordinate. For the first option, we obtain a generating function $F^{2143}(0,q-1,\gamma)$. For the second option, we go from $(\gamma_1,\gamma_1,q)$ to $(\gamma_2,\gamma_1{+}1,q)$ and trim the signature so the corresponding generating function is $F^{2143}(\gamma_1{+}1{-}\gamma_2,q,\gamma')$.

The main case is $|\gamma|\geq2$ and $k\geq1$. Our goal is to decrease $k$. As $k\geq1$, when we start at $(\gamma_1,\gamma_1{+}k,q)$, we are allowed to have an arbitrary number of $(\gamma_1,\gamma_1{+}k,q)$ first via unrecorded edges, which provide a factor of $s$, before we choose the next edge. Let's now compare $F^{2143}(k,q,\gamma)$ with $sF^{2143}(k{-}1,q,\gamma)$. The paths enumerated by each of them largely coincide, including those that decrease $q$ right away. The only exception is that paths that go directly from some number of $(\gamma_1,\gamma_1{+}k,q)$ to the next recorded edge ending at $(\gamma_2,\gamma_1{+}k{+}1,q)$ are counted by $F^{2143}(k,q,\gamma)$ but not by $sF^{2143}(k{-}1,q,\gamma)$; and similarly the paths that go directly to $(\gamma_2,\gamma_1{+}k,q)$ from $(\gamma_1,\gamma_1{+}k{-}1,q)$ are counted only by $sF^{2143}(k{-}1,q,\gamma)$. As a result,
\begin{align*}
&F^{2143}(k,q,\gamma)-sF^{2143}(k{-}1,q,\gamma)\\
=&sF^{2143}(\gamma_1{+}1{-}\gamma_2{+}k,q,\gamma')-sF^{2143}(\gamma_1{-}\gamma_2{+}k,q,\gamma')
\end{align*}
which is equivalent to the statement that we need.
\end{proof}

Notice that the recursive formula provided in Lemma~\ref{lem:2143gen} can determine $F^{2143}$ uniquely.

\begin{lemma}\label{lem:1234gen}
For $k\geq0$, $q\geq1$, $\gamma\in\Z^m$ with $m\geq1$, we have
\begin{equation*}
F^{1234}(k,q,\gamma)=\begin{cases}
s^k &|\gamma|=1,\\
F^{2143}(k,q,\gamma) &q=1,\\
F^{1234}(k,q{-}1,\gamma)+F^{1234}(\gamma_1{+}1{-}\gamma_2{+}k,q,\gamma') &|\gamma|\geq2,q\geq2.
\end{cases}
\end{equation*}
\end{lemma}
\begin{proof}
We refer the readers to the succession rule in Proposition~\ref{prop:suc1234}

If $|\gamma|=1$, then we are considering paths that start at $(\gamma_1,\gamma_1+k,q)$ with no recorded edges. Since every edge is recorded when $q\geq2$, our only option is to decrease $q$ all the way down to 1 and then use the succession rule of $(\gamma_1,\gamma_1+k,1)$. Now we can have an arbitrary number of $(\gamma_1,\gamma_1{+}k,1)$ followed by an arbitrary number of $(\gamma_1,\gamma_1{+}k{-}1,1)$ and so on up to an arbitrary number of $(\gamma_1,\gamma_1{+}1,1)$. The generating function is thus $(1+t+t^2+\cdots)^k=s^k$.

If $q=1$, the succession rules for $\P^{1234}$ and $\P^{2143}$ are the same so we have $\P^{1234}_{k,1,\gamma}=\P^{2143}_{k,1,\gamma}$. Therefore, $F^{1234}(k,1,\gamma)=F^{2143}(k,1,\gamma)$.

When $q\geq2$ and $|\gamma|\geq2$, for a path in $\P^{1234}_{k,q,\gamma}$, it starts at $(\gamma_1,\gamma_1{+}k,q)$. Since $q\geq2$, all edges that keep the same $z$-coordinate $q$ are recorded. So we have exactly two options: decrease $q$ by 1, which results in the generating function $F^{1234}(k,q{-}1,\gamma)$, and go to $(\gamma_2,\gamma_1+1,q)$ indicated by the signature $\gamma$, which results in the generating function $F^{1234}(\gamma_1{+}1{-}\gamma_2{+}k,q,\gamma')$. Take the sum and we get the desired equation.
\end{proof}

With sufficient tools to determine the generating functions $F^{2143}$ and $F^{1234}$, we are ready to obtain their equality.

\begin{lemma}\label{lem:2143gen=1234gen}
For $k\geq0$, $q\geq1$, $\gamma\in\Z^m$ with $m\geq1$, $$F^{1234}(k,q,\gamma)=F^{2143}(k,q,\gamma).$$
\end{lemma}
\begin{proof}
We proceed by induction on $|\gamma|$, $q$ and $k$ in this order. From Lemma~\ref{lem:2143gen} and Lemma~\ref{lem:1234gen}, our statement is true when $|\gamma|=1$ and is also true when $|\gamma|\geq2$ and $q=1$. When $|\gamma|\geq2$, $q\geq2$ and $k=0$, from Lemma~\ref{lem:2143gen},
$$F^{2143}(0,q,\gamma)=F^{2143}(0,q-1,\gamma)+F^{2143}(\gamma_1+1-\gamma_2,q,\gamma')$$
and from Lemma~\ref{lem:1234gen},
$$F^{1234}(0,q,\gamma)=F^{1234}(0,q-1,\gamma)+F^{1234}(\gamma_1+1-\gamma_2,q,\gamma')$$
so by induction hypothesis, $F^{2143}(0,q,\gamma)=F^{1234}(0,q,\gamma)$.

Now assume that $|\gamma|\geq2$, $q\geq2$ and $k\geq1$. With induction hypothesis and for the ease of notation, for the arguments that we already know the equality of $F^{1234}$ and $F^{2143}$, we will just write $F$ instead. By Lemma~\ref{lem:2143gen} and Lemma~\ref{lem:1234gen}, and by induction hypothesis, we have that
\begin{align*}
F^{2143}(k,q,\gamma)=&sF(k{-}1,q,\gamma)+sF(\gamma_1{+}1{-}\gamma_2{+}k,q,\gamma')-sF(\gamma_1{-}\gamma_2{+}k,q,\gamma')\\
=&sF(k{-}1,q{-}1,\gamma)+sF(\gamma_1{-}\gamma_2{+}k,q,\gamma')\\
&+sF(\gamma_1{+}1{-}\gamma_2{+}k,q{-}1,\gamma')+sF(\gamma_1{+}2{-}\gamma_3{+}k,q,\gamma'')\\
&-sF(\gamma_1{-}\gamma_2{+}k,q{-}1,\gamma')-sF(\gamma_1{+}1{-}\gamma_3{+}k,q,\gamma'')\\
=&F(k,q{-}1,\gamma)+F(\gamma_1{+}1{-}\gamma_2{+}k,q,\gamma')\\
=&F^{1234}(k,q,\gamma)
\end{align*}
as desired. We also see that the above argument goes through when $|\gamma'|=1$, in which case $\gamma''=\emptyset$. Therefore, the induction step is established so we obtain the desired lemma.
\end{proof}

With the main technical lemma (Lemma~\ref{lem:2143gen=1234gen}), Theorem~\ref{thm:main} becomes immediate.
\begin{proof}[Proof of Theorem~\ref{thm:main}]
For $\pi\in\{1234,2143\}$ and $j\leq n$, 
$$B_n^j(\pi)=\sum_{\gamma_1=j+1}[t^{n-j-|\gamma|+1}]F^{\pi}(0,j+1,\gamma).$$
Since $F^{1234}(0,j+1,\gamma)=F^{2143}(0,j+1,\gamma)$, $B_n^j(1234)=B_n^j(2143)$.
\end{proof}
\section{Open questions}\label{sec:open}

There are still many interesting questions to be asked. 

Firstly, the proof provided in Section~\ref{sec:proof} is semi-bijective. With recursive formulas provided in Lemma~\ref{lem:2143gen} and Lemma~\ref{lem:1234gen}, we are able to obtain the equality of $F^{1234}(k,q,\gamma)=F^{2143}(k,q,\gamma)$. However, is there an explicit bijection between paths in $\P^{1234}_{k,q,\gamma}$ and $\P^{2143}_{k,q,\gamma}$ that is length-preserving? 

Secondly, for a fixed $j\geq0$, it is desirable to obtain an explicit formula for the generating function $$\sum_{n=j}^{\infty}B_n^j(\pi)t^{n-j}$$ for either $\pi\in\{1234,2143\}$. The case $j=0$ is the generating function for 1234 (or 2143) avoiding permutations $\sum_{n}S_n(1234)t^n$, which is studied in \cite{bousquet2002four} and already has a complicated form.

Thirdly, can our techniques be further generalized? We make the following conjecture.
\begin{conjecture}\label{conj:12345}
For $j\leq n$, $|B_n^j(12345)|=|B_n^j(21354)|$.
\end{conjecture}
We have checked Conjecture~\ref{conj:12345} for $n\leq 7$. Notice that when $j=0$, the statement holds \cite{west1990permutations} and when $j=n$, it is not hard to see that both sides equal the Catalan number $C_j$.

More generally, it is known that the identity element $1,2,\ldots,k$ and $\pi=2,1,3,\ldots$, $k{-}2,k,k{-}1$ are Wilf equivalent in the sense of permutations \cite{west1990permutations}. So are they Wilf equivalent in signed permutations?

Finally, we note that some results in this paper fit nicely into the framework of $\mathrm{st}$-Wilf equivalence introduced by Sagan and Savage \cite{sagansavage}. In the classical case of the symmetric group $S_n$ and a permutation statistic $\mathrm{st}\colon S_n\to \mathcal{S}$ (where $\mathcal{S}$ is some set), the patterns $\pi_1,\pi_2$ are said to be $\mathrm{st}$-Wilf equivalent if for all $n\geq 1$ and $\sigma\in \mathcal{S}$, the number of $w\in S_n(\pi_1)$ with $\mathrm{st}(w)=\sigma$ is equal to the analogous number for $\pi_2$. The analogous definition for $B_n$ is clear; furthermore, our Theorem \ref{thm:main} and Lemma \ref{lem:2143gen=1234gen} can be restated as instances of $\mathrm{st}$-Wilf equivalence. But there are many other important permutation statistics on $B_n$, and proving corresponding $\mathrm{st}$-Wilf equivalences for permutation patterns in signed permutations is an interesting direction for further investigation.

\section*{Acknowledgments}
This research was carried out as part of the 2019 Summer Program in Undergraduate Research (SPUR) of the MIT Mathematics Department. The authors thank Prof.\,Alex Postnikov for suggesting the project and Christian Gaetz, Prof.\,Ankur Moitra and Prof.\,David Jerison for helpful conversations.

\bibliographystyle{plain}
\bibliography{ref}
\end{document}